\documentclass{siamltex}
\usepackage{amsfonts,amsmath,mathrsfs}
\usepackage{amssymb}
\usepackage{hyperref}
\usepackage{graphicx}
\usepackage{wrapfig}
\usepackage{subfig}
\usepackage{float}
\usepackage{bbm}

\newcommand{\N}{\mathbb{N}}
\newcommand{\R}{\mathbb{R}}
\newcommand{\C}{\mathbb{C}}

\newcommand{\Ecal}{\mathcal{E}}
\newcommand{\Vcal}{\mathcal{V}}
\newcommand{\Ical}{\mathcal{I}}
\newcommand{\Jcal}{\mathcal{J}}
\newcommand{\Ucal}{\mathcal{U}}
\newcommand{\Rcal}{\mathcal{R}}
\newcommand{\norm}[1]{\left\| #1 \right\|}
\newtheorem{remark}[theorem]{Remark}
\newtheorem{algorithm}{Algorithm}

\newcommand{\overbar}[1]{\mkern 1.0mu\overline{\mkern-1.0mu#1\mkern-1.0mu}\mkern 1.0mu}

\title{Optimizing electrode positions in electrical impedance tomography}

\author{Nuutti Hyv\"onen\footnotemark[2]
\and Aku Sepp\"anen\footnotemark[3]
\and Stratos Staboulis\footnotemark[2]
}

\begin{document}
\maketitle

\renewcommand{\thefootnote}{\fnsymbol{footnote}}

\footnotetext[2]{Aalto University, Department of Mathematics and Systems Analysis, P.O. Box 11100, FI-00076 Aalto, Finland (nuutti.hyvonen@aalto.fi, stratos.staboulis@aalto.fi). The work of N.~Hyv\"onen and 
S.~Staboulis was supported by the Academy of Finland (decision 141044).}
\footnotetext[3]{University of Eastern Finland, Department of Applied Physics, FI-70211 Kuopio, Finland (aku.seppanen@uef.fi). The work of Aku Sepp\"anen was supported by the Academy of Finland (the Centre of Excellence in Inverse Problems Research and decisions 270174 and 273536).}

\renewcommand{\thefootnote}{\arabic{footnote}}

\begin{abstract}
Electrical impedance tomography is an imaging modality for recovering information about the conductivity inside a physical body from boundary measurements of current and voltage. In practice, such measurements are performed with a finite number of contact electrodes. This work considers finding  optimal positions for the electrodes within the Bayesian paradigm based on available prior information on the conductivity; the aim is to place the electrodes so that the posterior density of the (discretized) conductivity, i.e., the conditional density of the conductivity given the measurements, is as localized as possible. To make such an approach computationally feasible, the complete electrode forward model of impedance tomography is linearized around the prior expectation of the conductivity, allowing explicit representation for the (approximate) posterior covariance matrix. Two approaches are considered: minimizing the trace or the determinant of the posterior covariance. The introduced optimization algorithm is of the steepest descent type, with the needed gradients computed based on appropriate Fr\'echet derivatives of the complete electrode model. The functionality of the methodology is demonstrated via two-dimensional numerical experiments.
\end{abstract}

\begin{keywords}
Electrical impedance tomography, optimal electrode locations, Bayesian inversion, complete electrode model, optimal experiment design
\end{keywords}

\begin{AMS}
65N21, 35Q60, 62F15
\end{AMS}

\pagestyle{myheadings}
\thispagestyle{plain}
\markboth{N. HYV\"ONEN, A. SEPP\"ANEN, AND S. STABOULIS}{OPTIMIZING ELECTRODE POSITIONS IN EIT}


\section{Introduction}
\label{sec:intro}
{\em Electrical impedance tomography} (EIT) is an imaging modality for recovering information 
about the electrical conductivity inside a physical body from boundary measurements of current and potential. 
In practice, such measurements are performed with a finite number of contact electrodes. 
The reconstruction problem of EIT is a highly nonlinear and illposed inverse problem. 
For more information on the theory and practice of EIT, we refer to the review articles
\cite{Borcea02,Cheney99,Uhlmann09} and the references therein.

The research on {\em optimal experiment design} in EIT has mostly focused on
determining optimal current injection patterns.
The most well-known approach to optimizing current injections 
is based on the {\em distinguishability} criterion
\cite{Isaacson1986},
i.e.,~maximizing the norm of the difference between the electrode potentials 
corresponding to the unknown true conductivity and a known reference conductivity distribution.
Several variants of the distinguishability approach have been proposed; see,~e.g.,~\cite{Koksal1995,Lionheart2001} for versions 
with constraints on the  injected currents. 
The application of the method to
planar electrode arrays was considered in \cite{Kao2003}.
The distinguishability criterion leads to the use of current patterns
generated by exciting several electrodes simultaneously.
For other studies where the sensitivity of the EIT measurements is controlled by
injecting currents through several electrodes at a time, see
\cite{Polydorides2002,Yan2006}.
In geophysical applications of EIT,
the data is often collected using four-point measurements;
choosing the optimal ones
for different electrode settings was studied in
\cite{alHagrey2002,Furman2007,Stummer2004,Wilkinson2006}.
In all the above-cited works, the optimal experiment setup was considered in the
deterministic inversion framework.
The {\em Bayesian} approach to 
selecting optimal current patterns was studied in \cite{Kaipio04,Kaipio07}.
In the Bayesian experiment design \cite{Atkinson2007,Chaloner1995},
(statistical) prior information on the unknown is taken into account in
optimizing the measurements.

In addition to choosing the electrode currents,
the sensitivity of EIT measurements can also be controlled by
varying the electrode configuration.
For studies on comparing different 
setups,
see \cite{Graham2007,Nebuya2006,Noordegraaf1996,Paulson1992,Tidswell2003}.
{\em Optimizing} the electrode locations in EIT, however, has been studied only recently 
in \cite{Haber10}.
In this study,
a large set of {\em point electrodes} was set in predefined locations,
and an optimization method with sparsity constraints for the current injections and potential measurements 
was applied to select a feasible set of electrodes. 
The number of active electrodes was not fixed. See also \cite{Khambete99} for a study on the optimal placement of {\em four} electrodes in impedance pneumography.

In the present paper, 
the problem of optimizing the electrode locations is considered in a more realistic setting
than in \cite{Haber10}.
We model the EIT measurements with the {\em complete electrode model} (CEM) \cite{Cheng89}, 
which takes into account the electrode shapes and the contact impedances at the electrode-boundary interfaces.
We aim at finding optimal locations for the 
finite-sized electrodes.
Unlike in \cite{Haber10},
the admissible electrode locations are not limited to a finite set of predefined points.
On the other hand, the number of electrodes {\em is} predefined.
These attributes are appealing from a practical point of view
because in many laboratories/clinics/field surveys,
the number of electrodes is limited by the specifications of the measurement device,
while the possibilities of arranging the electrodes are almost unlimited.
As in \cite{Kaipio04,Kaipio07},
the optimal experiments are considered in the Bayesian inversion framework to enable the incorporation of prior information on the conductivity.
Given a prior probability density for the (discretized) conductivity, reflecting the knowledge 
about the interior of the examined object before the measurements, 
the aim is to place the electrodes so that the posterior density of the conductivity, i.e., 
the conditional density of the conductivity given the measurements, is as localized as possible 
(when marginalized over all possible measurements). 
To be more precise, the considered design criteria are 
the A- and D-optimality (see,~e.g.,~\cite{Atkinson2007,Chaloner1995}).
Allowing simplifications, the former corresponds to the minimization of the trace 
of the posterior covariance and the latter to the maximization of the information gain 
when the prior is replaced by the posterior. 

To make our approach computationally feasible, 
we linearize the measurement map of the CEM around the prior expectation of the conductivity, 
which allows an explicit representation for the posterior covariance and thus also for the objective functions corresponding to the A- and D-optimality criteria. 
Since the introduced optimization algorithm is of the steepest descent type, 
it requires numerical computation of the derivatives for the linearized measurement map 
with respect to the electrode locations, that is, of certain second order (shape) derivatives for the CEM. 
We perform the needed differentiations by resorting to the appropriate Fr\'echet derivatives of the 
(non-discretized) CEM (cf.~\cite{Darde12,Darde13a}); 
in addition to reducing the computational load, 
this leads to higher stability compared to perturbation-based numerical differentiation schemes. 

The functionality of the chosen methodology is demonstrated via two-dimensional numerical experiments. The conclusions of our tests are three-fold: (i) The introduced optimization algorithm seems functional, that is, it finds the electrode locations satisfying the considered optimality criteria in settings where the global optimum can be determined by testing all possible cases. (ii) The prior information on the conductivity considerably affects the optimal electrode locations in many relevant settings. Furthermore, the optimal electrode locations may be nonuniform even if the prior information on the conductivity is homogeneous over the examined object. The extent of this latter effect depends heavily on the complexity of the object shape. (iii) Choosing optimal electrode locations results in improved solutions to the (nonlinear) inverse problem of EIT --- at least, for our Bayesian reconstruction algorithm (see, e.g., \cite{Darde13a}) and if the target conductivity is drawn from the assumed prior density. 

This text is organized as follows. Section~\ref{sec:CEM} recalls the CEM and introduces the needed Fr\'echet derivatives. The principles of Bayesian inversion and optimal experiment design are considered in Section~\ref{sec:Bayes} and the implementation of the optimization algorithm in Section~\ref{sec:alg}. Finally, the numerical results are presented in Section~\ref{sec:numer} and the conclusions listed in Section~\ref{sec:conc}. 

\section{Complete electrode model}
\label{sec:CEM}
We start by recalling the CEM of EIT. Subsequently, we introduce the Fr\'echet derivatives of the (linearized) current-to-voltage map of the CEM needed for optimizing the electrode locations in the following sections.

\subsection{Forward problem}

In practical EIT, $M \geq 2$ contact electrodes $\{ E_m\}_{m=1}^M$ are attached to the exterior surface of a body $\Omega$. A net current $ I_m\in\C $ is driven through each $ E_m $ and the resulting constant electrode potentials $ U = [U_1,\ldots,U_M]^{\rm T} \in \C^M $ are measured.
Because of charge conservation, any applicable current pattern $ I = [I_1,\ldots,I_M]^{\rm T} $ belongs to the subspace
\[ 
\C^M_\diamond = \Big\{V\in\C^M\,\Big|\, \sum_{m=1}^M V_m = 0\Big\}.
\]
The contact impedances at the electrode-object interfaces are modeled by a vector $z = [z_1,\dots, z_M]^{\rm T} \in \C^M$ whose components are assumed to satisfy 
\begin{equation}
\label{eq:z}
{\rm Re} (z_m) \geq c, \quad m=1, \dots, M,
\end{equation}
for some constant $c>0$.

We assume that $\Omega \subset \R^n$, $n=2,3$, is a bounded domain with a smooth boundary. Moreover, the electrodes $\{ E_m\}_{m=1}^M$ are identified with open, connected, smooth, nonempty subsets of $\partial \Omega$ and assumed to be  mutually well separated, i.e., $\overbar{E}_k \cap \overbar{E}_l = \emptyset$ for $k \not= l$. We denote $E = \cup E_m$. The mathematical model that most accurately predicts real-life EIT measurements is the CEM \cite{Cheng89}: The electromagnetic potential $u$ inside $\Omega$ and the potentials on the electrodes $U$ satisfy
\begin{equation}
\label{eq:cemeqs}
\begin{array}{ll}
\displaystyle{\nabla \cdot\sigma\nabla u = 0 \qquad}  &{\rm in}\;\; \Omega, \\[6pt] 
{\displaystyle{\nu\cdot\sigma\nabla u} = 0 }\qquad &{\rm on}\;\;\partial\Omega\setminus\overbar{E},\\[6pt] 
{\displaystyle u+z_m{\nu\cdot\sigma\nabla u} = U_m } \qquad &{\rm on}\;\; E_m, \quad m=1, \dots, M, \\[2pt] 
{\displaystyle \int_{E_m}\nu\cdot\sigma\nabla u\,{\rm d}S} = I_m, \qquad & m=1,\ldots,M, \\[4pt]
\end{array}
\end{equation}
interpreted in the weak sense. Here, $ \nu = \nu(x) $ denotes the exterior unit normal of $ \partial\Omega $. Moreover, the (real) symmetric admittivity distribution $\sigma\in L^{\infty}(\Omega, \C^{n \times n})$ that characterizes the electric properties of the medium is assumed to satisfy
\begin{equation}
\label{eq:sigma}
{\rm Re} (\sigma \xi \cdot \overbar{\xi}) \geq c | \xi |^2,
\qquad c > 0,
\end{equation}
for all $\xi \in \C^n$ almost everywhere in $\Omega$. 
A physical justification of \eqref{eq:cemeqs} can be found in \cite{Cheng89}.

Given an input current pattern $ I\in \C^M_\diamond $,
an admittivity $\sigma$ and contact impedances $z$ with the properties \eqref{eq:sigma} and \eqref{eq:z}, respectively, the potential pair $ (u,U) \in H^1(\Omega) \oplus \C^M $ is uniquely determined by \eqref{eq:cemeqs} up to the ground level of potential. This can be shown by considering the Hilbert space $ \mathbb{H}^1 :=  (H^1(\Omega)\oplus \C^M)/\C $ with the norm
\begin{equation*}
\norm{(v,V)}_{\mathbb{H}^1} = \inf_{c\in\C}\Big\{ \norm{v-c}_{H^1(\Omega)}^2 + \sum_{m=1}^M|V_m-c|^2 \Big\}^{1/2}
\end{equation*}
and the variational formulation of \eqref{eq:cemeqs} given by \cite{Somersalo92} 
\begin{equation}\label{eq:varcem}
B_\sigma\big((u,U),(v,V)\big)  \,=  \, I\cdot \overbar{V} \qquad {\rm for} \ {\rm all} \ (v,V) \in  \mathbb{H}^1,
\end{equation}
where
$$
B_\sigma\big((u,U),(v,V)\big) = \int_\Omega \sigma\nabla u\cdot \nabla \overbar{v} \,{\rm d}x + \sum_{m=1}^M \frac{1}{z_m}\int_{E_m}(u-U_m)(\overbar{v}-\overbar{V}_m)\,{\rm d}S.
$$
Since $B_\sigma: \mathbb{H}^1 \times \mathbb{H}^1$ is continuous and coercive \cite{Somersalo92,Hyvonen04}, it follows easily from the Lax--Milgram theorem that \eqref{eq:varcem} is uniquely solvable. Moreover, the solution pair $(u,U) \in \mathbb{H}^1$ depends continuously on the data,
\begin{equation}
\label{eq:bound}
\| (u,U) \|_{\mathbb{H}^1} \leq C  |I|,
\end{equation}
where $C = C(\Omega,E, \sigma, z) > 0$. 

The measurement, or current-to-voltage map of the CEM is defined via
\begin{equation}\label{eq:measmat}
R: I \mapsto U, \quad \C^M_\diamond \to \C^M / \C .
\end{equation} 
Due to an obvious symmetry of \eqref{eq:varcem}, $R$ can be represented as a symmetric complex $(M-1)\times(M-1)$ matrix (with respect to any chosen basis for $\C^M_\diamond \sim \C^M/\C)$.

\subsection{Linearization of the forward model}
Next, we consider the linearization of the map $\sigma \mapsto (u(\sigma),U(\sigma))$ at a fixed current pattern $I \in \C^M_\diamond$. To this end, let us define the 
set of admissible conductivities,
$$
\Sigma := \big\{ \sigma \in L^\infty(\Omega, \C^{n \times n}) \ | \ \sigma^{\rm T} = \sigma \text{ and } \eqref{eq:sigma} \text{ holds with some constant } c > 0 \big\},
$$
and the space of conductivity perturbations,
$$
K := \{ \kappa \in L^\infty(\Omega, \C^{n \times n}) \ | \ \kappa^{\rm T} = \kappa \}.
$$
It is well known that $(u(\sigma),U(\sigma))$ is Fr\'echet differentiable with respect to $\sigma$: For a fixed $\sigma \in \Sigma$ and any small enough $\kappa \in K$ in the $L^\infty$-topology, it holds that (cf., e.g., \cite{KaipioSomersalo})
\begin{equation}
\label{eq:cond_deriv}
\big\| \big(u(\sigma + \kappa ), U(\sigma + \kappa)\big) - \big(u(\sigma),U(\sigma)\big) - \big(u'(\sigma;\kappa), U'(\sigma;\kappa)\big) \big\|_{\mathbb{H}^1} = O\big(\| \kappa \|_{L^{\infty}(\Omega)}^2\big) |I|,
\end{equation}
where $(u'(\sigma;\kappa), U'(\sigma;\kappa)) \in \mathbb{H}^1$ is the unique solution of the variational problem
\begin{equation}
\label{eq:vardercem}
B_\sigma\big((u',U'),(v,V)\big) = - \int_{\Omega} \kappa \nabla u(\sigma) \cdot \nabla \overbar{v} \, {\rm d} x \qquad {\rm for} \ {\rm all} \ (v,V) \in  \mathbb{H}^1.
\end{equation}
In other words, the Fr\'echet derivative of $\sigma \mapsto (u(\sigma),U(\sigma)) \in \mathbb{H}^1$ at some $\sigma \in \Sigma$ is the linear map $K \ni \kappa \mapsto (u'(\sigma;\kappa), U'(\sigma;\kappa)) \in \mathbb{H}^1$. Notice that the unique solvability of \eqref{eq:vardercem} is a straightforward consequence of the Lax--Milgram theorem due to the continuity and coercivity of $B_\sigma$.

If the support of the perturbation $\kappa$ is restricted to a compact subset of $\Omega$ and the conductivity $\sigma$ exhibits some extra regularity in some neighborhood of $\partial \Omega$, it can be shown that the norm on the left-hand side of \eqref{eq:cond_deriv} may be replaced with a stronger one. With this aim in mind, we define two new concepts: the space of admissible conductivities with H\"older boundary smoothness,
$$
\Sigma_{\partial \Omega}^{k,\beta} := \{ \sigma \in \Sigma \ | \ 
\sigma|_{G \cap \Omega} \in  C^{k,\beta}(\overline{G \cap \Omega}) \text{ for some open }  G \supset \partial \Omega \}, \quad k \in \N_0, \ 0 \leq \beta \leq 1,
$$ 
and the space of compactly supported conductivity perturbations, 
$$
K_\delta := \{\kappa \in K \ | \ {\rm dist}( {\rm \supp}\, \kappa, \partial \Omega) \geq \delta \},
$$
where $\delta > 0$.

\begin{theorem}
\label{thm:efrechet}
Let $\sigma \in \Sigma_{\partial \Omega}^{0,1}$ and $\delta > 0$ be fixed. Then, there exists a smooth domain $\Omega_0 \Subset \Omega$ such that
$$ 
\big\| \big(u(\sigma + \kappa ), \ U(\sigma + \kappa)\big) - \big(u(\sigma),U(\sigma)\big) - \big(u'(\sigma;\kappa), \ U'(\sigma;\kappa)\big) \big\|_{\mathbb{H}^s(\Omega\setminus \overbar{\Omega}_0)} \! = \!O\big(\| \kappa \|_{L^{\infty}(\Omega)}^2\big) |I|
$$
for all small enough $\kappa \in K_\delta$ in the topology of $L^\infty(\Omega)$ and any fixed $s < 2$. Here we denote $\mathbb{H}^s(\Omega\setminus \overbar{\Omega}_0) = (H^{s}(\Omega\setminus \overbar{\Omega}_0) \oplus \C^M) / \C$.
\end{theorem}

\begin{proof}
Obviously, there exist smooth domains $\Omega_0$ and $\Omega_1$ such that $\Omega_1 \Subset \Omega_0 \Subset \Omega$, $\sigma|_{\Omega \setminus \overbar{\Omega}_1} \in C^{0,1}(\overline{\Omega} \setminus \Omega_1)$ and ${\rm supp}\, \kappa \subset \overbar{\Omega}_1$ for all $\kappa \in K_\delta$.

Let us denote
$$
(u^\kappa, U^\kappa) = \big( u(\sigma + \kappa ), U(\sigma + \kappa) \big) - \big(u(\sigma),U(\sigma)\big) - \big(u'(\sigma;\kappa), U'(\sigma;\kappa)\big) \in \mathbb{H}^1.
$$
It follows by a straightforward calculation from \eqref{eq:varcem} and \eqref{eq:vardercem} that
\begin{equation}
\label{eq:varkappacem}
B_\sigma\big((u^\kappa, U^\kappa), (v,V)\big) = \int_{\Omega}\kappa \nabla \big(u(\sigma) - u(\sigma+\kappa)\big) \cdot \nabla \overbar{v} \, {\rm d} x \qquad {\rm for} \ {\rm all} \ (v,V) \in  \mathbb{H}^1.
\end{equation}
For any compactly supported test function $v \in C_0^\infty(\Omega \setminus \overbar{\Omega}_1)$, the right hand side of \eqref{eq:varkappacem} and the second term of $B_\sigma((u^\kappa, U^\kappa), (v,V))$ vanish, which means that
\begin{equation}
\label{eq:kappa_basic}
\nabla \cdot (\sigma \nabla u^\kappa) = 0 \quad {\rm in} \ \Omega \setminus \overbar{\Omega}_1
\end{equation}
in the sense of distributions.

Resorting to the same techniques as used in~\cite{Somersalo92} when proving the equivalence of \eqref{eq:cemeqs} and \eqref{eq:varcem}, it follows easily from \eqref{eq:varkappacem} that altogether $(u^\kappa, U^\kappa)$  satisfies in $\Omega \setminus \overbar{\Omega}_0$ the boundary value problem
\begin{equation}
\label{eq:kappacemeqs}
\begin{array}{ll}
\displaystyle{\nabla \cdot(\sigma\nabla u^\kappa) = 0 \qquad}  &{\rm in}\;\; \Omega \setminus \overbar{\Omega}_0, \\[6pt] 
{\displaystyle{\nu\cdot\sigma\nabla u^\kappa} = 0 }\qquad &{\rm on}\;\;\partial\Omega\setminus\overbar{E},\\[6pt] 
{\displaystyle u^\kappa+z_m{\nu\cdot\sigma\nabla u^\kappa} = U^\kappa_m } \qquad &{\rm on}\;\; E_m, \quad m=1, \dots, M, \\[6pt] 
{\displaystyle \nu\cdot\sigma\nabla  u^\kappa = g^\kappa } \qquad &{\rm on}\;\; \partial \Omega_0, \\[2pt]
{\displaystyle \int_{E_m}\nu\cdot\sigma\nabla u^\kappa\,{\rm d}S} = 0, \qquad & m=1,\ldots,M, \\[4pt]
\end{array}
\end{equation}
where $g^\kappa \in H^{-1/2}(\partial \Omega_0)$ is the Neumann trace of $u^\kappa \in H^1(\Omega)/ \C$, which is well defined by virtue of \eqref{eq:kappa_basic} (cf.~\cite{Lions72}). Since $\sigma$ is Lipschitz continuous in $\Omega \setminus \overbar{\Omega}_1$, the interior regularity theory for elliptic partial differential equations \cite{Grisvard85} yields that $u^\kappa \in H^2(G)/ \C$ for any (smooth) open domain $G \Subset \Omega \setminus \overbar{\Omega}_1$ and, furthermore, 
$$
\| u^\kappa \|_{H^{2}(G_0)/\C} \leq C 
\| u^\kappa \|_{H^{1}(G)/\C}
$$
for any other (fixed) smooth domain $G_0 \Subset G$.
Choosing $G$ and $G_0$ to be open neighborhoods of $\partial \Omega_0$, the trace theorem finally gives the estimate
\begin{equation}
\label{eq:g_cont}
\| g^\kappa \|_{H^{1/2}(\partial \Omega_0)} \leq C \| u^\kappa \|_{H^{2}(G_0)/\C} \leq C 
\| u^\kappa \|_{H^{1}(G)/\C} = O\big(\| \kappa \|_{L^{\infty}(\Omega)}^2\big) |I|,
\end{equation}
where 
the last inequality is a weaker version of \eqref{eq:cond_deriv}.

Let us then define $f^\kappa$ to be the Neumann trace of $u^\kappa$ on $\partial \Omega$. Exactly as in \cite[proof of Lemma~2.1]{Hyvonen09}, we get
\begin{equation}
\label{eq:f_cont}
\| f^\kappa \|_{L^2(\partial \Omega)} \leq  C \| (u^\kappa, U^\kappa) \|_{\mathbb{H}^1} \leq  O\big(\| \kappa \|_{L^{\infty}(\Omega)}^2\big) |I|, 
\end{equation}
where the second step is just \eqref{eq:cond_deriv}.
In particular, the trace theorem and the continuous dependence on the Neumann data for the first equation of \eqref{eq:kappacemeqs} indicate that (cf.~\cite{Grisvard85})
\begin{align}
\label{eq:ukappab}
\| u^\kappa \|_{H^1(\partial \Omega)/\C} &\leq C \| u^\kappa \|_{H^{3/2}(\Omega \setminus \overbar{\Omega}_0)/ \C} \leq C \big( \|f^\kappa \|_{L^2(\partial \Omega)} + \|g^\kappa \|_{L^2(\partial \Omega_0)} \big) \nonumber \\[1mm]
&\leq O\big(\| \kappa \|_{L^{\infty}(\Omega)}^2\big) |I|,
\end{align}
where the last step is obtained by combining \eqref{eq:g_cont} and \eqref{eq:f_cont}. Via a bootstrap type argument, we may now use the properties of zero continuation in Sobolev spaces \cite{Lions72} and the third equation of \eqref{eq:kappacemeqs} combined with \eqref{eq:z} to deduce for any $s < 2$ that
\begin{align}
\label{eq:f_cont2}
\| f^\kappa \|_{H^{s-3/2}(\partial \Omega)} & \leq C \| f^\kappa \|_{H^1(E)} \leq
C \Big( \sum_{m=1}^m \|U_m^\kappa - u^\kappa \|^2_{H^1(E_m)}\Big)^{1/2} \nonumber \\  
&\leq  O\big(\| \kappa \|_{L^{\infty}(\Omega)}^2\big) |I|,
\end{align}
where the last inequality follows from \eqref{eq:ukappab} and \eqref{eq:cond_deriv} as in \cite[proof of Lemma~3.1]{Hanke11}.
The (re)employment of the continuous dependence on the Neumann data for the first equation of \eqref{eq:kappacemeqs} shows that
$$
\| u^\kappa \|_{H^s(\Omega \setminus \overbar{\Omega}_0)/\C} \leq C \big( \|f^\kappa \|_{H^{s-3/2}(\partial \Omega)} + \|g^\kappa \|_{H^{s-3/2}(\partial \Omega_0)} \big) \leq  O\big(\| \kappa \|_{L^{\infty}(\Omega)}^2\big) |I|
$$
due to \eqref{eq:f_cont2} and \eqref{eq:g_cont}. Combining this with \eqref{eq:cond_deriv}, the assertion easily follows.
\end{proof}

\begin{remark}
\label{remark}
By the same argument, one can also easily prove that 
$$
\| \big(u(\sigma),U(\sigma)\big) \|_{\mathbb{H}^s(\Omega\setminus \overbar{\Omega}_0)} \leq C |I|
$$
and 
$$
\| \big(u'(\sigma;\kappa), U'(\sigma;\kappa)\big) \|_{\mathbb{H}^s(\Omega\setminus \overbar{\Omega}_0)} \leq  C \| \kappa \|_{L^{\infty}(\Omega)}|I|
$$
for any $s < 2$ under the assumptions of Theorem~\ref{thm:efrechet}.
\end{remark}


It is obvious that the measurement map $R$ of the CEM, originally defined by \eqref{eq:measmat}, can be treated as an operator of two variables, the conductivity and the input current, that is,
\begin{equation}
\label{eq:measmat2}
R: (\sigma, I) \mapsto U(\sigma), \quad \Sigma \times \C^M_\diamond \to \C^M / \C.
\end{equation}
It follows trivially from \eqref{eq:cond_deriv} that $R$ is Fr\'echet differentiable with respect to the conductivity and the corresponding (bilinear) derivative
\begin{equation}
\label{eq:sigma_der}
\partial_\sigma R[\sigma]: K \times \C^M_\diamond \to \C^M / \C
\end{equation}
is defined by
$$
\partial_\sigma R[\sigma] (\kappa, I) = U'(\sigma; \kappa),
$$
where $U'(\sigma, \kappa)$ is the second part of the solution to \eqref{eq:vardercem}, with the underlying current pattern $I \in \C^M_\diamond$.

\subsection{Electrode shape derivatives}
We start by recalling from \cite{Darde12} a general way of perturbing the electrodes $\{ E_m \}_{m=1}^M$ with the help of $C^1$ vector fields living
on $\partial E$.
We denote
$$
\mathcal{B}_d \, = \, \{ a\in C^1(\partial E, \R^n) \ | \ 
\| a \|_{C^1(\partial E, \R^n)} < d \}
$$
and let $B(x,d) = \{ y \in \R^n \ | \ |y - x| < d \}$ be a standard open
ball in $\R^n$. Notice that in two spatial dimensions $\partial E$ consist only of the end points of the electrodes; in this degenerate case, we simply define $C^1(\partial E, \R^2)$ to be the space of $2M$-tuples of vectors supported, respectively, at those end points, with the corresponding norm defined, e.g., as the sum of the norms of the individual vectors.
For small enough $d > 0$,
$$
P_x:  B(x,d) \to \partial \Omega
$$
denotes the (nonlinear) projection that maps $y \in  B(x,d)$, which lies sufficiently close
to $x \in \partial E$, in the direction of  $\nu(x)$ onto $\partial \Omega$.
To make this definition unambiguous, it is also required that $P_x x = x$ and $P_x$ is continuous. It is rather obvious that $P_x$ is well defined for some 
$d = d(\Omega) > 0$ that can be chosen independently of $x \in \partial E$ 
due to a compactness argument. For each $a \in \mathcal{B}_d$, we introduce a perturbed set of `electrode boundaries',
\begin{equation}
\label{Eboundary}
\partial E_m^a \, = \, \{ z \in \partial \Omega \ | \ z = P_x(x + a(x)) \
{\rm for} \ {\rm some} \ x \in \partial E_m \}, \qquad m = 1, \dots , M. 
\end{equation}
According to \cite[Proposition 3.1]{Darde12}, there exist $d > 0$ such that for any  $a \in \mathcal{B}_d$, the formula \eqref{Eboundary} defines a set of well separated bounded and connected electrodes $E^a_1, \dots, E^a_M \subset \partial \Omega$ with $C^1$ boundaries. 

As a consequence, the measurement map of CEM, introduced originally in
\eqref{eq:measmat} and fine-tuned by \eqref{eq:measmat2}, can be further extended to be an operator of three variables, 
\begin{equation}
\label{eq:measmat3}
R: (a,\sigma, I) \mapsto U(a,\sigma), \quad \mathcal{B}_d \times \Sigma \times \C^M_{\diamond} \to \C^M / \C,
\end{equation}
where $U(a,\sigma)$ is the electrode potential component of the solution
$(u(a,\sigma), U(a,\sigma)) \in \mathbb{H}^1$ to (\ref{eq:cemeqs}) when the original electrodes
$\{ E_m \}_{m=1}^M$ are replaced by the perturbed ones $\{ E_m^a \}_{m=1}^M$.

By assuming some extra smoothness for the conductivity, i.e., interpreting
$$
R: \mathcal{B}_d \times \Sigma_{\partial \Omega}^{1,0} \times \C^M_{\diamond} \to \C^M / \C,
$$
it can be shown (cf.~\cite[Theorem~4.1]{Darde12}) that $R$ is Fr\'echet differentiable with respect to the first variable at the origin. The corresponding derivative 
$$
\partial_a R[0]= \partial_a R: C^1(\partial E, \R^n) \times \Sigma_{\partial \Omega}^{1,0}  \times \C^M_{\diamond} \to \C^M / \C,
$$ 
which is linear with respect to the first and the third variable,
can be sampled as indicated by the following proposition that is a slight generalization of \cite[Corollary~4.2]{Darde12}.

\begin{proposition}
\label{Ederivative}
Let $(\tilde{u}, \tilde{U}) \in \mathbb{H}^1$ be the solution of 
\eqref{eq:cemeqs} for a given electrode current pattern $\tilde{I} \in 
\C^M_{\diamond}$ and a conductivity $\sigma \in \Sigma_{\partial \Omega}^{1,0}$. Then, for any $a \in C^1(\partial E, \R^n)$ and  $I \in \C^M_{\diamond}$, it holds that
\begin{equation}
\label{eq:ederivative}
\partial_a R(a,\sigma,I) \cdot \tilde{I} = 
- \sum_{m=1}^M \frac{1}{z_m}\int_{\partial E_m}(a \cdot \nu_{\partial E_m})(U_m-u)
(\tilde{U}_m-\tilde{u}) \, {\rm d} s,
\end{equation}
where $\nu_{\partial E_m}$ is the exterior unit normal of $\partial E_m$ in the 
tangent bundle of $\partial \Omega$ and $(u,U) \in \mathbb{H}^1$ is the solution of \eqref{eq:cemeqs} for $I \in  \C^M_{\diamond}$ and $\sigma$.
\end{proposition}

\begin{proof}
The proof of the proposition is essentially the same as for \cite[Corollary~4.2]{Darde12}, where $\sigma$ was assumed to be smooth. However, it is straightforward to check that $C^{1}$ regularity close to $\partial \Omega$ is actually sufficient.
\end{proof}

In the following section, we need to linearize the measurement map with respect to $\sigma$ in order to obtain a computationally feasible measure for the optimality of the electrode locations in the Bayesian framework. For this reason, we would like to find a formula for the Fr\'echet derivative of $\partial_\sigma R$ defined in \eqref{eq:sigma_der} with respect to a electrode perturbation field $a \in C^1(\partial E, \R^n)$. To avoid further technicalities, we settle for computing the derivatives in the reverse order without proving  symmetry of second derivatives.

\begin{theorem}
\label{secder}
The operator $\partial_a R: C^1(\partial E, \R^n) \times \Sigma_{\partial \Omega}^{1,0}  \times \C^M_{\diamond} \to \C^M / \C$ is Fr\'echet differentiable with respect to its second variable in the sense that
$$
\| \partial_a R(a,\sigma + \kappa, I) - \partial_a R(a,\sigma, I)
- \partial_\sigma \partial_a R[\sigma] (a,\kappa, I) \|_{\C^M /\C}
= O\big(\| \kappa \|_{L^{\infty}(\Omega)}^2\big) |I| \| a \|_{C^1(\partial E)}
$$
for all small enough $\kappa \in K_\delta$ with $\delta > 0$ fixed. The (tri)linear second derivative
$$
\partial_\sigma\partial_a R[\sigma]: C^1(\partial E, \R^n) \times K_\delta  \times \C^M_{\diamond} \to \C^M / \C
$$
at $\sigma \in \Sigma_{\partial \Omega}^{1,0}$ is defined as follows: For any $a \in C^1(\partial E, \R^n)$, $\kappa \in K_\delta$ and $I \in \C^M_{\diamond}$,
\begin{align}
\label{eq:esderivative}
\partial_\sigma \partial_a R[\sigma](a,\kappa,I) \cdot \tilde{I} &= 
- \sum_{m=1}^M \frac{1}{z_m} \int_{\partial E_m}(a \cdot \nu_{\partial E_m})\big((U'_m-u') (\tilde{U}_m-\tilde{u}) \nonumber \\
& \qquad \qquad \qquad \quad \  + (U_m-u) (\tilde{U}'_m-\tilde{u}')\big) {\rm d} s ,
\end{align}
where $(u(\sigma), U(\sigma)), (\tilde{u}(\sigma), \tilde{U}(\sigma))\in \mathbb{H}^1$  are the solutions of \eqref{eq:cemeqs} for $I, \tilde{I} \in  \C^M_{\diamond}$, respectively, and $(u'(\sigma;\kappa), U'(\sigma;\kappa)), (\tilde{u}'(\sigma;\kappa), \tilde{U}'(\sigma;\kappa)) \in \mathbb{H}^1$ are those of \eqref{eq:vardercem}.
\end{theorem}

\begin{proof}
To begin with, we note that \eqref{eq:esderivative} is a proper definition of a (bounded, trilinear)  operator from $C^1(\partial E, \R^n) \times K_\delta  \times \C^M_{\diamond}$ to $\C^M / \C$ since $\C^M / \C$ is finite-dimensional and its dual is realized by $\C^M_\diamond$.~(Notice also that the right-hand side of \eqref{eq:esderivative} is well defined and depends boundedly on $a$, $\kappa$ and $I$ by virtue of Remark~\ref{remark}, the trace theorem and the Schwarz inequality; cf.~the estimates below.)

By the triangle inequality  it holds for any $c \in \C$ that
\begin{align}
\label{eq:apu}
\big\| & \big( U_m(\sigma + \kappa ) - u(\sigma + \kappa ) \big) - \big(U_m(\sigma ) - u(\sigma  ) \big) -  \big(U_m'(\sigma;\kappa) - u'(\sigma;\kappa)\big)
\big\|_{L^2} \\[1mm]
& \leq \big\| U_m(\sigma + \kappa ) - U_m(\sigma) -  U_m'(\sigma;\kappa) - c \big\|_{L^2} + \big\| u(\sigma + \kappa ) - u(\sigma) -  u'(\sigma;\kappa) - c \big\|_{L^2} \nonumber
\end{align}
where the $L^2$-norms are taken over $\partial E_m$ for an arbitrary $m=1,\dots, M$.
Furthermore, two applications of the trace theorem induce the estimate
\begin{equation}
\label{eq:apu2}
\big\| u(\sigma + \kappa ) - u(\sigma) -  u'(\sigma;\kappa) - c \big\|_{L^2(\partial E)} \!\leq
C  \| u(\sigma + \kappa ) - u(\sigma) -  u'(\sigma;\kappa) - c \big\|_{H^{1+\eta}(G)},
\end{equation}
where $\eta > 0$, the subset $G \subset \Omega$ is an arbitrary interior neighborhood of $\partial \Omega$ and $C = C(\eta,G,E) > 0$.
Since $\kappa \in K_\delta$ and $\sigma \in \Sigma_{\partial \Omega}^{1,0}$ satisfy the assumptions of Theorem~\ref{thm:efrechet}, combining \eqref{eq:apu} with \eqref{eq:apu2} and taking the infimum over $c \in \C$ yields
\begin{equation}
\label{eq:apu3}
\big\|  \big( U_m(\sigma + \kappa ) - u(\sigma + \kappa ) \big) - \big(U_m(\sigma ) - u(\sigma  ) \big) -  \big(U_m'(\sigma;\kappa) - u'(\sigma;\kappa)\big)
\big\|_{L^2}
\leq O\big(\| \kappa \|_{L^{\infty}(\Omega)}^2\big) |I|,
\end{equation}
where the $L^2$-norm is once again taken over $\partial E_m$. Obviously, an analogous estimate is valid when $I$ is replaced by $\tilde{I}$ and the potentials, together with their derivatives, are changed accordingly.

Using the sampling formulas \eqref{eq:ederivative} and  \eqref{eq:esderivative} together with \eqref{eq:apu3} and the Schwarz inequality, it can easily be deduced that 
$$
\big | \big(\partial_a R'(a,\sigma + \kappa,I) - \partial_a R'(a,\sigma,I) - \partial_\sigma \partial_a R'[\sigma](a,\kappa,I) \big) \cdot \tilde{I} \big| \leq O\big(\| \kappa \|_{L^{\infty}(\Omega)}^2\big) |I| |\tilde{I}|\| a \|_{\infty},
$$
for all $\tilde{I} \in \C^M_\diamond$. Since $\C^M_\diamond$ is the dual of $\C^M / \C$, this completes the proof.
\end{proof}

\section{Bayesian inversion and optimal electrode positions}
\label{sec:Bayes}
In the Bayesian approach \cite{KaipioSomersalo} to inverse problems all parameters that exhibit uncertainty are modelled as random variables. Each quantity of interest is given a {\em prior} probability distribution which reflects (a part of) the available information about it before the measurement. The measurement itself is modelled as a realization of a compound random variable depending on,~e.g.,~model parameters and noise. Under suitable assumptions, by using the Bayes' formula for conditional probability, one obtains the {\em posterior} probability density in which the updated information about the parameters of interest is encoded. The practical problem is to develop numerical methods for exploring the posterior distribution. In this section we revisit these concepts in the framework of the CEM, the aim being to derive the desired posterior covariance related optimality criteria for the electrode positions. 

By means of discretization, let us model the conductivity by a finite dimensional real-valued random variable and denote its generic realization by $\sigma \in \R^m$. Since in most applications it is reasonable to inject several linearly independent electrode currents $\{I^{(j)}\}_{j=1}^{N} \subset \R_\diamond^M$ into $\Omega$ and measure the resulting (noisy) electrode voltages $\{V^{(j)}\}_{j=1}^N \subset \R^M$, we employ the shorthand notations 
\[
\Ical = [(I^{(1)})^{\rm T}, (I^{(2)})^{\rm T}, \ldots, (I^{(N)})^{\rm T}]^{\rm T}, \quad \Vcal = [(V^{(1)})^{\rm T}, (V^{(2)})^{\rm T}, \ldots, (V^{(N)})^{\rm T}]^{\rm T}
\]
for the total current injection and electrode potential patterns, respectively. Here the measurement $\Vcal$ is modelled as a realization of
a random variable that takes values in $\R^{MN}$ and depends intimately on the forward solution 
\[
\Ucal(\sigma) = \Rcal(\sigma,\Ical) = [R(\sigma,I^{(1)})^{\rm T}, R(\sigma,I^{(2)})^{\rm T}, \ldots, R(\sigma,I^{(N)})^{\rm T}]^{\rm T},
\] 
where the current-to-voltage map $R(\cdot, \cdot)$ is defined as in \eqref{eq:measmat2}. Since we are interested in considering the electrode configuration as a variable, we denote an admissible set of electrodes by $\Ecal = \{E_1,E_2,\ldots, E_M\}$. 

Suppose that our prior information on the conductivity is encoded in the probability density $p_{\rm pr}(\sigma)$ and the {\em likelihood} density $p(\Vcal|\sigma; \Ical,\Ecal)$ is known. According to the Bayes' formula, the posterior density for $\sigma$ is given by
\begin{equation}
\label{eq:posterior}
p(\sigma|\Vcal;\Ical,\Ecal) = \frac{p(\Vcal|\sigma;\Ical,\Ecal)\, p_{\rm pr}(\sigma)}{p(\Vcal;\Ical,\Ecal)},
\end{equation}
where the density of $\Vcal$ is obtained via marginalization:
$$
p(\Vcal;\Ical,\Ecal) = \int p(\sigma,\Vcal;\Ical,\Ecal)\, {\rm d} \sigma 
= \int \, p(\Vcal|\sigma;\Ical,\Ecal) \, p_{\rm pr}(\sigma)\,{\rm d} \sigma .
$$
In particular, the posterior density can be used to define different point estimates such as the (possibly non-unique) {\em maximum a posteriori} (MAP) estimate
\begin{equation}\label{map}
\hat{\sigma}_{\rm MAP}(\Vcal;\Ical,\Ecal) = \arg \max_{\sigma} \, p(\sigma|\Vcal;\Ical,\Ecal)
\end{equation}
and the {\em conditional mean} (CM) estimate
\begin{equation*}
\hat{\sigma}_{\rm CM}(\Vcal;\Ical,\Ecal) = \int \sigma p(\sigma|\Vcal;\Ical,\Ecal)\, {\rm d}\sigma.
\end{equation*}
Often, these (and other) point estimates cannot be obtained as closed form solutions to some deterministic problems, but instead, they require implementation of, e.g., 
{\em Monte Carlo} (MC) type sampling methods or high-dimensional numerical optimization schemes~\cite{Kaipio2000,KaipioSomersalo}. 

Next we move on to consider criteria for optimal experiment design. 
Our general approach is to form a suitable functional reflecting the feasibility of $\Ecal$ and define its minimizer/maximizer to be the {\em optimal electrode positions}.
From the point estimate perspective, the most intuitive choice for the to-be-minimized functional is, arguably, of the mean square error type \cite{Kaipio04,Kaipio07}, 
\begin{align}\label{eq:mse}
\psi: (\Ical,\Ecal)  \mapsto & \int \left(\int | A (\sigma - \hat{\sigma}(\Vcal;\Ical,\Ecal))|^2 p(\sigma|\Vcal;\Ical,\Ecal)\, {\rm d}\sigma \right) p(\Vcal;\Ical,\Ecal)\, {\rm d}\Vcal  \\ 
&\,= {\rm tr}\left(\int\int A(\sigma - \hat{\sigma}(\Vcal;\Ical,\Ecal))(\sigma - \hat{\sigma}(\Vcal;\Ical,\Ecal))^{\rm T}A^{\rm T} p(\sigma,\Vcal;\Ical,\Ecal)\,{\rm d}\sigma {\rm d}\Vcal\right) \nonumber,
\end{align}
where $A$ is the chosen weight matrix and $\hat{\sigma}$ is the point estimate of interest. (Although the numerical experiments in Section~\ref{sec:numer} only consider optimization of the electrode locations, in this section $\psi$ is interpreted as a function of both $\Ical$ and $\Ecal$, which reflects the fact that the applied current patterns and the electrode positions could in principle be optimized simultaneously.) 
In optimal experiment design \cite{Atkinson2007,Chaloner1995}, the minimization of the functional \eqref{eq:mse} is often called the A-optimality criterion; intuitively, the A-optimal design minimizes the variation of $\sigma$  around the considered point estimate $\hat{\sigma}$ in the (semi)norm induced by the positive (semi)definite matrix $A^{\rm T} A$.

Another approach is to compare the prior and posterior distributions directly without committing oneself to a specific point estimate. 
As an example, the maximization of the information gain when the prior is replaced with the posterior leads to the Kullback--Leibler divergence of the prior from the posterior,
\begin{equation}\label{eq:kl}
\psi : (\Ical,\Ecal) \mapsto \int \left(\int \log\left(\frac{p(\sigma|\Vcal;\Ical,\Ecal)}{p_{\rm pr}(\sigma)}\right)p(\sigma|\Vcal;\Ical,\Ecal)\,{\rm d}\sigma\right) p(\Vcal;\Ical,\Ecal) \, {\rm d} \Vcal.
\end{equation}
In optimal experiment design the maximization of the functional \eqref{eq:kl} is known as the D-optimality \cite{Atkinson2007,Chaloner1995}. Note that by the Fubini's theorem and the Bayes' formula we have  
\begin{align}\label{eq:kls}
\int \log (p_{\rm pr}(\sigma)) \, p(\sigma, \Vcal;\Ical,\Ecal) \, {\rm d}\sigma {\rm d}\Vcal &= \int \left(\int  p(\Vcal|\sigma;\Ical,\Ecal) \, {\rm d}\Vcal\right) \log (p_{\rm pr}(\sigma))\,p_{\rm pr}(\sigma)\, {\rm d}\sigma \nonumber \\
& = \int \log(p_{\rm pr}(\sigma))\,p_{\rm pr}(\sigma)\, {\rm d}\sigma,
\end{align}
which is independent of the pair $(\Ical,\Ecal)$. In consequence, $p_{\rm pr}(\sigma)$ can be dropped from the denominator in \eqref{eq:kl} without affecting the maximizer with respect to $(\Ical,\Ecal)$.

If the measurement is of the form $\Vcal = \varphi(\Ucal(\sigma),\varepsilon)$, where $\varepsilon$ models the noise and $\varphi$ is differentiable, then the results of Section~\ref{sec:CEM} could in principle be used to build a gradient based optimization algorithm for minimizing \eqref{eq:mse} or maximizing \eqref{eq:kl}. In particular, the most commonly used additive noise models fall into this framework. However, such an approach could easily end up being extremely expensive computationally due to the extensive MC sampling required for evaluating the feasibility functional as well as its derivatives. For this reason, we move on to derive closed form expressions for  \eqref{eq:mse} and \eqref{eq:kl} in case the prior and the noise process are Gaussian, and the current-to-potential map is linearized with respect to the conductivity.

\subsection{Gaussian models with linearization}

For both clarity and simplicity,  we omit writing the $(\Ical,\Ecal)$-dependence explicitly in the following. 
We choose an additive noise model
\begin{equation}\label{measurement}
\Vcal = \Ucal(\sigma) + \varepsilon,
\end{equation}
where $\varepsilon \in \R^{MN}$ is a realization of a zero mean Gaussian random variable with the covariance matrix $\Gamma_{\rm noise}$. Assuming that the prior is also Gaussian with the mean $\sigma_\ast$ and the covariance matrix $\Gamma_{\rm pr}$, it follows from \eqref{eq:posterior} that the posterior density satisfies
\begin{equation}\label{eq:gausspost}
p(\sigma|\Vcal) \propto \exp \left( -\frac{1}{2}(\Ucal(\sigma)-\Vcal)^{\rm T}\Gamma_{\rm noise}^{-1}(\Ucal(\sigma)-\Vcal) -\frac{1}{2} (\sigma-\sigma_\ast)^{\rm T}\Gamma_{\rm pr}^{-1}(\sigma-\sigma_\ast)\right),
\end{equation}
where the constant of proportionality is independent of $\sigma$ but depends in general on $\Vcal$, $\Ical$ and $\Ecal$. We also assume that both of the needed inverse covariance matrices exist. 

In order to evaluate the integrals in \eqref{eq:mse} and \eqref{eq:kl} explicitly, we linearize the current-to-voltage map centered at the prior mean, i.e., we apply
\begin{equation}\label{eq:J0}
\Ucal(\sigma) \approx \Ucal(\sigma_\ast) + \Jcal_\ast(\sigma-\sigma_\ast),
\end{equation}
where $\Jcal_\ast := \Jcal(\sigma_\ast)$ is the matrix representation of $ \kappa \mapsto \Ucal'(\sigma_\ast;\kappa) = \partial_\sigma \Rcal[\sigma_\ast](\kappa,\Ical)$ (cf.~\eqref{eq:sigma_der}). As a result, we obtain an approximate posterior density 
\begin{equation}\label{eq:p0}
p_\ast(\sigma|\Vcal) \propto \exp\left( -\frac{1}{2}(\Jcal_\ast\sigma - \Vcal_\ast)^{\rm T}\Gamma_{\rm noise}^{-1}(\Jcal_\ast\sigma - \Vcal_\ast) -\frac{1}{2} (\sigma-\sigma_\ast)^{\rm T}\Gamma_{\rm pr}^{-1}(\sigma-\sigma_\ast) \right),
\end{equation}
where $\Vcal_\ast := \Vcal - \Ucal(\sigma_\ast) + \Jcal_\ast \sigma_\ast $. 
Notice that $p_\ast(\, \cdot \,|\Vcal)$ is a product of two multivariate normal densities and thus a multivariate Gaussian itself. By completing the squares with respect to $\sigma$, the covariance matrix and the mean of $p_\ast(\, \cdot \, | \Vcal)$ can be written as
\begin{equation}\label{eq:Gamma0}
\Gamma_\ast = \big(\Jcal_\ast^{\rm T}\Gamma_{\rm noise}^{-1}\Jcal_\ast + \Gamma_{\rm pr}^{-1}\big)^{-1} \qquad {\rm and}
\qquad \hat{\sigma}_\ast = \Gamma_\ast \Jcal_\ast^{\rm T} \Gamma_{\rm noise}^{-1} \Vcal_\ast, 
\end{equation}
respectively \cite{KaipioSomersalo}.
In particular, this means that we altogether have
$$
p_\ast(\sigma|\Vcal) =
\frac{1}{\sqrt{ (2 \pi)^m \det(\Gamma_\ast)} } \exp\left( -\frac{1}{2} (\sigma - \hat{\sigma}_\ast)^{\rm T} \Gamma_\ast^{-1} (\sigma - \hat{\sigma}_\ast) \right)
$$
due to the normalization requirement of probability densities.

Let us choose $\hat{\sigma}(\Vcal) = \hat{\sigma}_\ast$ as our point estimate of interest. Replacing $p(\sigma | \Vcal)$ by $p_\ast(\sigma|\Vcal)$ and $p(\Vcal)$ by the density of $\Vcal$ corresponding to the linearized model, say, $p_\ast(\Vcal)$ in \eqref{eq:mse}, we get a new, simpler feasibility functional
\begin{align}\label{eq:msegauss}
\psi: (\Ical, \Ecal) \mapsto  & \ {\rm tr} \left( \int 
\left(\int  
A(\sigma - \hat{\sigma}_\ast)(\sigma - \hat{\sigma}_\ast)^{\rm T}A^{\rm T} p_\ast(\sigma|\Vcal)\, {\rm d}\sigma 
\right) 
p_\ast(\Vcal)\, {\rm d}\Vcal\right) \nonumber \\
& \, = {\rm tr}\left(\int A\Gamma_\ast A^{\rm T} p_\ast(\Vcal) \, {\rm d}\Vcal\right) = {\rm tr}\big(A\Gamma_\ast A^{\rm T}\big),
\end{align}
where we used the independence of $\Gamma_\ast$ from  $\Vcal$. Take note that the right-hand side of \eqref{eq:msegauss} can be evaluated with a reasonable computational effort. 
Performing the same simplifications in \eqref{eq:kl}, we get another computationally attractive functional:
\begin{align}\label{eq:klgauss}
\psi: (\Ical, \Ecal) \ \mapsto  & \int \left(\int \log\left(\frac{p_\ast(\sigma|\Vcal)}{p_{\rm pr}(\sigma)}\right)p_\ast(\sigma|\Vcal) {\rm d}\sigma\right) p_\ast(\Vcal){\rm d}\Vcal \nonumber \\
& \, = \int \left(\int \log\left(p_\ast(\sigma|\Vcal)\right)p_\ast(\sigma|\Vcal) {\rm d}\sigma\right) p_\ast(\Vcal)\,{\rm d}\Vcal + C \nonumber  \\
& \, = - \int \log \left( \sqrt{(2 \pi e)^m  \det(\Gamma_\ast)} \right) p_\ast(\Vcal) \, {\rm d} \Vcal + C \nonumber \\
& \, = -\frac{1}{2} \log \det(\Gamma_\ast) + C',
\end{align}
where the first step follows from the logic \eqref{eq:kls} and the second one from the known form of the {\em differential entropy} for a multivariate Gaussian (cf.~\cite{Lazo78}). Furthermore, $C$ and $C'$ are scalars that are independent of the current pattern and the electrodes. Hence, for the linearized model, maximizing the information gain is equivalent to minimizing $\log \det (\Gamma_\ast)$. 


\section{Algorithmic implementation}
\label{sec:alg}
In this section we introduce our electrode position optimization algorithm in two spatial dimensions. To this end, suppose that $\Omega$ is star-shaped and can thus be parametrized by a $2\pi$-periodic simple closed curve $\gamma \colon \R \to \R^2$ with respect to the polar angle.
Each electrode in the configuration $\Ecal = \{E_1,E_2,\ldots, E_M\}$ is composed of an open arc segment on $\partial\Omega$, 
meaning that $E_m$ is determined by the pair of its extremal polar angles $\theta_m^- < \theta_m^+ $. We denote the full angular parameter by $\theta = [\theta^-, \theta^+]^{\rm T}\in \R^{2M}$ where $\theta^\pm = [\theta^\pm_1,\theta^\pm_2,\ldots,\theta^\pm_M]^{\rm T}$. Given that a particular $\theta$ defines an admissible electrode configuration, we may denote the dependence of any function on the electrodes via this parameter. 

In order to build a gradient based optimization algorithm for \eqref{eq:msegauss} or \eqref{eq:klgauss}, we (approximately) calculate the derivatives ${\partial\Jcal}/{\partial\theta_m^\pm}$ (cf.~\eqref{eq:J0} and \eqref{eq:Gamma0}) by applying the reverse order second derivative formula \eqref{eq:esderivative} in the form
\begin{equation}\label{eq:tsderivative}
\begin{split}
\tilde{\Ical} \cdot \frac{\partial^2\Ucal}{\partial\sigma_k\partial\theta_m^\pm} = &\mp \sum_{j=1}^N |\dot{\gamma}\circ \gamma^{-1}|\bigg( \frac{\partial U_m^{(j)}}{\partial\sigma_k} - \frac{\partial  u^{(j)} }{\partial\sigma_k} \bigg)\big(\tilde{U}^{(j)} - \tilde{u}^{(j)} \big)\Bigg|_{\gamma(\theta_m^\pm)} \\
& \mp \sum_{j=1}^N |\dot{\gamma}\circ \gamma^{-1}| \big(U^{(j)} - u^{(j)} \big)
\bigg( \frac{\partial \tilde{U}_m^{(j)}}{\partial\sigma_k} - \frac{\partial \tilde{u}^{(j)}}{\partial\sigma_k} \bigg)\Bigg|_{\gamma(\theta_m^\pm)},
\end{split}
\end{equation}
where $\dot{\gamma}$ is the derivative of $\gamma$. The pairs $(u^{(j)},U^{(j)})$ and $(\tilde{u}^{(j)},\tilde{U}^{(j)})$ are the solutions to \eqref{eq:cemeqs} for the $j$th current inputs $I^{(j)}$ and $\tilde{I}^{(j)}$ in the associated total current vectors $\Ical$ and $\tilde{\Ical}$, respectively. Furthermore, $((\partial u^{(j)})/(\partial \sigma_k),(\partial U^{(j)})/(\partial \sigma_k))$ and $((\partial \tilde{u}^{(j)})/(\partial \sigma_k),(\partial \tilde{U}^{(j)})/(\partial \sigma_k))$ are the corresponding solutions to \eqref{eq:vardercem} with the conductivity perturbation $\kappa$ being the $k$th component of the discretized $\sigma$. Notice that the integrals over the boundaries of the electrodes in \eqref{eq:esderivative} are reduced to point evaluations at the electrode edges in \eqref{eq:tsderivative}, which is the correct interpretation in two dimensions (cf.~\cite{Darde12}), and the choice $a \cdot \nu_{\partial E_m} = \pm |\dot{\gamma}\circ \gamma^{-1}|$ results in the derivatives with respect to the polar angles of the electrode end points.

In order to make use of the differentiation formula \eqref{eq:tsderivative} in practice, we are forced to carry out several technicalities. First of all, approximate solutions to \eqref{eq:cemeqs} are computed using piecewise linear {\em finite elements} (FE) in a polygonal domain $\Omega_{\rm poly}\approx \Omega$ with a triangulation $\mathcal{T}$. Since 
$\Omega_{\rm poly}$ and especially $\mathcal{T}$ can vary significantly depending on the electrode locations, we choose a fixed `background domain' $D \supset \Omega,\Omega_{\rm poly}$ with a homogeneous triangulation to act as a storage for an `extended conductivity' $\varsigma$ which is parametrized using piecewise linear basis functions on $D$. The values of $\varsigma$ are carried over to $\Omega_{\rm poly}$ by a projection matrix $P$, and the projected conductivity $P \varsigma$ is used in the approximation of $(u^{(j)}(\sigma),U^{(j)}(\sigma))$ and $(\tilde{u}^{(j)}(\sigma),\tilde{U}^{(j)}(\sigma))$ with $\sigma = \varsigma|_\Omega$; for detailed instructions on the assembly of the FE scheme for the CEM, see,~e.g.,~\cite{Vauhkonen99}. Secondly, we note that the FE approximations are also differentiable with respect to the variable $P\varsigma$, and the associated derivatives are determined by a formula analogous to \eqref{eq:vardercem} with the variational space replaced by the FE space in question. Finally, we employ the sampling formula \eqref{eq:tsderivative} with the exact solutions on the right-hand side replaced by their FE counterparts  
even though there exists no proof that such an approximation converges to the desired quantity when the FE discretization gets finer (cf.~\cite[Remark~2.4]{Darde13b}) --- in fact, based on Proposition~\ref{Ederivative}, we do not even know that the continuum derivatives exist if the support of the conductivity perturbation touches the boundary. However, these theoretical imperfections did not affect the stability of the numerical experiments in Section~\ref{sec:numer}.
To sum up, using the above scheme, we obtain a numerical approximation for $ \Jcal$ as well as for ${\partial\Jcal}/{\partial\theta_m^\pm}$ (cf.~\eqref{eq:J0}). 

In the following we assume that the electrode widths are fixed. Extending the proposed approach to the case where the electrode sizes are optimized instead/in addition to the electrode positions is a straightforward task. This assumption implies a dependence $\theta_m^+ = \theta_m^+(\theta_m^-)$. Differentiating the arc length formula, one obtains
\begin{equation}\label{eq:leibniz}
\frac{{\rm d}\theta_m^+}{{\rm d}\theta_m^-} = \frac{|\dot{\gamma}(\theta_m^-)|}{|\dot{\gamma}(\theta_m^+)|},
\end{equation}
which can be incorporated into all needed gradients via the chain rule. 
After this observation, we are finally ready to introduce our numerical algorithm for optimizing the feasibility functionals \eqref{eq:msegauss} and \eqref{eq:klgauss} with respect to the electrode positions.
\bigskip

\begin{algorithm}\label{alg:steep}
A steepest descent algorithm for finding the optimal electrode positions:
\vspace{12pt}
\begin{itemize}

\item[(0)] Fix the initial set of constant parameters: $z$ (contact impedances), $\Ical$ (electrode currents), $D$ and $\Omega$ (background domain and object), $\sigma_\ast$ (prior mean), $\Gamma_{\rm pr}$ (prior covariance) and $\Gamma_{\rm noise}$ (noise covariance). Choose the staring point for the iteration $\theta^- = \theta_{\rm init}^-$. 
\\[2pt]

\item[(1)] Select the desired cost functional; our choice is (cf.~\eqref{eq:msegauss} and \eqref{eq:klgauss})
\begin{equation}\label{eq:cost}
\psi(\theta^-) = \alpha\sum_{m=1}^M\frac{1}{g_m(\theta^-)} + \begin{cases} {\rm tr}(\Gamma_\ast(\theta^-)), \quad {\it or} \\[2pt] \log\det(\Gamma_\ast(\theta^-)), \;\; \end{cases} 
\end{equation}
where $\alpha > 0$ is a manually picked constant and $g_m$ is the length of the gap between the $m$th and $(m+1)$th electrode (modulo $M$). The first term in \eqref{eq:cost} is included to prevent the electrodes from getting too close together.
\\[2pt]

\item[(2)] Evaluate $\psi(\theta^-)$ and use the above numerical scheme (based on \eqref{eq:tsderivative} and \eqref{eq:leibniz}) to compute each $\partial\Jcal/\partial\theta_m^-$. Subsequently, by applying well-known differentiation formulas to \eqref{eq:Gamma0} and \eqref{eq:cost}, calculate each $\partial\psi/\partial\theta_m^-$ and build the gradient $ \nabla \psi (\theta^-) $. The new iterate is defined as
\[
\theta^-_{\rm new} 
:= \theta^- - t_{\rm min} \frac{\nabla\psi(\theta^-)}{|\nabla\psi(\theta^-)|}, \qquad  
t_{\rm min} := \arg \min_{t\in \triangle} \, \psi\bigg(\theta^- - t\frac{\nabla\psi(\theta^-)}{|\nabla\psi(\theta^-)|}\bigg),
\]
where the minimum is sought by a line search routine and $\triangle \subset [0,\infty)$ is chosen so that the gaps between the electrodes remain positive. Take note that within the line search, each evaluation of $\psi$ (cf.~\eqref{eq:cost}) requires construction of a FE mesh corresponding to the electrode positions specified by the input.
\\[2pt]

\item[(3)] Unless satisfactory convergence is observed, set $\theta^- = \theta^-_{\rm new}$, generate a corresponding $\Omega_{\rm poly}$, and repeat phase (2).
\\[2pt]

\item[(4)] Return $\theta^-_{\rm opt} = \theta^-_{\rm new}$.

\end{itemize}
\end{algorithm}
\bigskip 

In our numerical experiments, we set exclusively $\alpha = 10^{-4}$, but most of the results would be qualitatively the same even if $\alpha=0$. However, a positive $\alpha$ increases the speed of convergence exhibited by Algorithm~\ref{alg:steep}. We have chosen the weight matrix $A$ of the A-optimality condition to be the identity matrix in \eqref{eq:cost}, which means that all node values of the conductivity on the uniform triangulation of the background domain $D$ are considered equally important. Note also that in the floating point arithmetic there is a problem of overflow/underflow when evaluating determinants of large matrices. This can fortunately be circumvented by using the Cholesky decomposition $\Gamma_\ast=L L^{\rm T}$ of the symmetric and positive definite posterior covariance matrix. Indeed, 
\[
\log \det (\Gamma_\ast) = 2 \log \det(L) = 2\sum_i \log (l_{ii}),
\]
where $ l_{ii}$ are the diagonal entries of the triangular Cholesky factor $L$. This trick stabilizes the treatment of the second cost function in \eqref{eq:cost} considerably. 

\begin{figure}[t!]
\begin{center}
\includegraphics[width=0.75\textwidth]{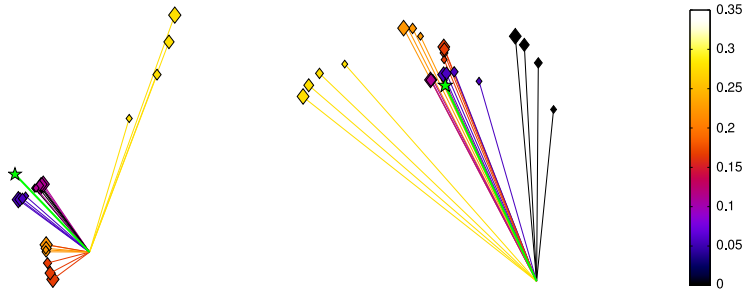} 
\caption{\label{fig:numvsana} Comparison of search directions for Algorithm~\ref{alg:steep} produced by different numerical methods. The vectors illustrate numerical approximations of $\nabla \psi (\theta^-) \in \R^2$ (cf.~\eqref{eq:cost}). The green one (with a star marker) is the Fr\'echet derivative whereas the rest of the vectors are computed using central difference formulas with different number of grid points and perturbation sizes. The colorbar visualizes the magnitude of the perturbation in the electrode positions and the bigger the marker size, the higher the order  of the central difference is ($2,4,6$ or $8$). The left-hand image corresponds to  $\log \det (\Gamma_\ast)$ and the right-hand image to ${\rm tr}(\Gamma_\ast)$ (the plots are in different scales).}
\end{center}
\vspace{-2mm}
\end{figure}

Let us conclude this section with a numerical motivation of the proposed (Fr\'echet) differentiation technique for the posterior covariance related functionals in \eqref{eq:cost}; cf.~Theorem~\ref{secder}. We claim that, besides clearly being computationally more efficient, the Fr\'echet derivative based method is also likely to be more accurate than a carelessly chosen perturbative numerical differentiation approach. In order to make a single comparison, we choose the following attributes for the test object: $\Omega = B(0,1)$, i.e., the unit disk, and $\sigma_\ast \equiv 1$. We employ two electrodes of width $\pi/16$ characterized by $\theta^-_{\rm init} = [0,\pi/2]^{\rm T}$ whence there is essentially only one current injection pattern $I = \Ical = [-1,1]^{\rm T}$. The contact impedances are set to $z_1=z_2=1$. The Gaussian prior and the additive zero mean noise process have the covariances 
\[
\Gamma_{\rm pr} = 0.2^2 \mathbbm{1}, \qquad \Gamma_{\rm noise} = \left(10^{-3}|U_1(\sigma_\ast,\theta^-) - U_2(\sigma_\ast,\theta^-)|\right)^2 \mathbbm{1},
\] 
respectively, where $\mathbbm{1}$ denotes an identity matrix of the appropriate size. Fig.~\ref{fig:numvsana} illustrates the line search directions at the first step of Algorithm~\ref{alg:steep} obtained by the Fr\'echet derivative and by a number of central difference formulas \cite{Mathews99}. Depending on the order of the employed difference rule and the size of the used perturbation for the electrode locations, the search directions provided by the perturbative method vary considerably. On the other hand, the directions given by the Fr\'echet derivative technique seem to be in a relatively good agreement with the difference schemes of the highest order with relatively small (but not too small) perturbation sizes, which arguably indicates both computational efficiency and reliability of the Fr\'echet derivatives.

\section{Numerical examples}
\label{sec:numer}

In this section, as a proof of concept, we apply the proposed algorithm to a few test cases. In all examples we choose, for simplicity, the prior mean to be homogeneous, $\sigma_\ast \equiv 1$, and fix all contact impedances to the unit value regardless of the number of electrodes in use. Furthermore, we use the fixed, full set of current input patterns $\Ical = [e_1^{\rm T},e_1^{\rm T},\ldots, e_1^{\rm T}]^{\rm T}-[e_2^{\rm T},e_3^{\rm T},\ldots, e_M^{\rm T}]^{\rm T}\in \R^{M(M-1)}$ where ${\rm e}_m$ denotes the $m$th Cartesian basis vector. In other words, we compose a basis of $\R_\diamond^M$ by fixing one feeding electrode and letting the current exit in turns from the remaining ones. A bit surprisingly, the choice of the current patterns did not have a notable effect on the optimal electrode positions in our numerical tests; according to our experience, this would not necessarily be the case for considerably higher noise levels. Note, however, that if a full set of $M-1$ independent current injections was not available, the choice of current injection patterns would have an important role in the optimal experiment design \cite{Kaipio04}.
We restrict our attention to Gaussian smoothness priors with covariance matrices $\Gamma_{\rm pr} = \Gamma(\lambda,\kappa)$ of the form
\begin{equation}\label{eq:smooth}
\Gamma_{ij} = \kappa_{ij}^2 \exp\bigg(-\frac{|x^{(i)} - x^{(j)}|^2}{2\lambda^2}\bigg), \qquad x^{(i)},x^{(j)} \in D,
\end{equation}
where $\kappa_{ij}^2 > 0$ are the covariance factors between the node values of the conductivity on the uniform triangulation of  the background domain $D$, and $\lambda$ is the correlation length that controls spatial smoothness.
If $\kappa_{ij} = \kappa$ is constant, the prior is said to be homogeneous. In particular, bear in mind that the prior mean is homogeneous in all our experiments, and the term ``inhomogeneous prior'' reflects a property of the prior covariance matrix.  


In what follows, we consider three test cases. In the first one, we choose a simple inhomogeneous prior and use only four electrodes so that the optimized electrode positions can be verified by brute force simulations; the motive is to validate the functionality of Algorithm~\ref{alg:steep}. The second example continues to use an intuitive inhomogeneous prior but with twelve electrodes, making the brute force computations practically infeasible; the idea is to demonstrate that the output of Algorithm \ref{alg:steep} remains as expected even for a higher dimensional setting. The last example considers still twelve electrodes, but with three different object shapes and a homogeneous prior; the aim is to indicate that the domain shape also has an effect on the optimal electrode locations. In all three tests, we assume mean-free, additive white noise model with the covariance matrix 
$$
\Gamma_{\rm noise} = \big(10^{-3} \max_{k,l} |\Ucal_k(\sigma_\ast, \theta^-_{\rm init})-\Ucal_l(\sigma_\ast, \theta^-_{\rm init},)|\big)^2 \, \mathbbm{1},
$$
where $\Ucal(\sigma_\ast, \theta^-_{\rm init})$ refers to the (exact) total electrode potential vector corresponding to the prior mean conductivity and the initial electrode configuration. In practice, the noise might also depend on the electrode configuration, but we choose to ignore this complication. (Note that $\Gamma_{\rm noise}$ changes slightly from test to test as  $\Ucal(\sigma_\ast, \theta^-_{\rm init})$ does.)

\subsection*{Case 1: A circular subdomain with high prior variance}

\begin{figure}[t!]
\begin{center}
\includegraphics[width=1\textwidth]{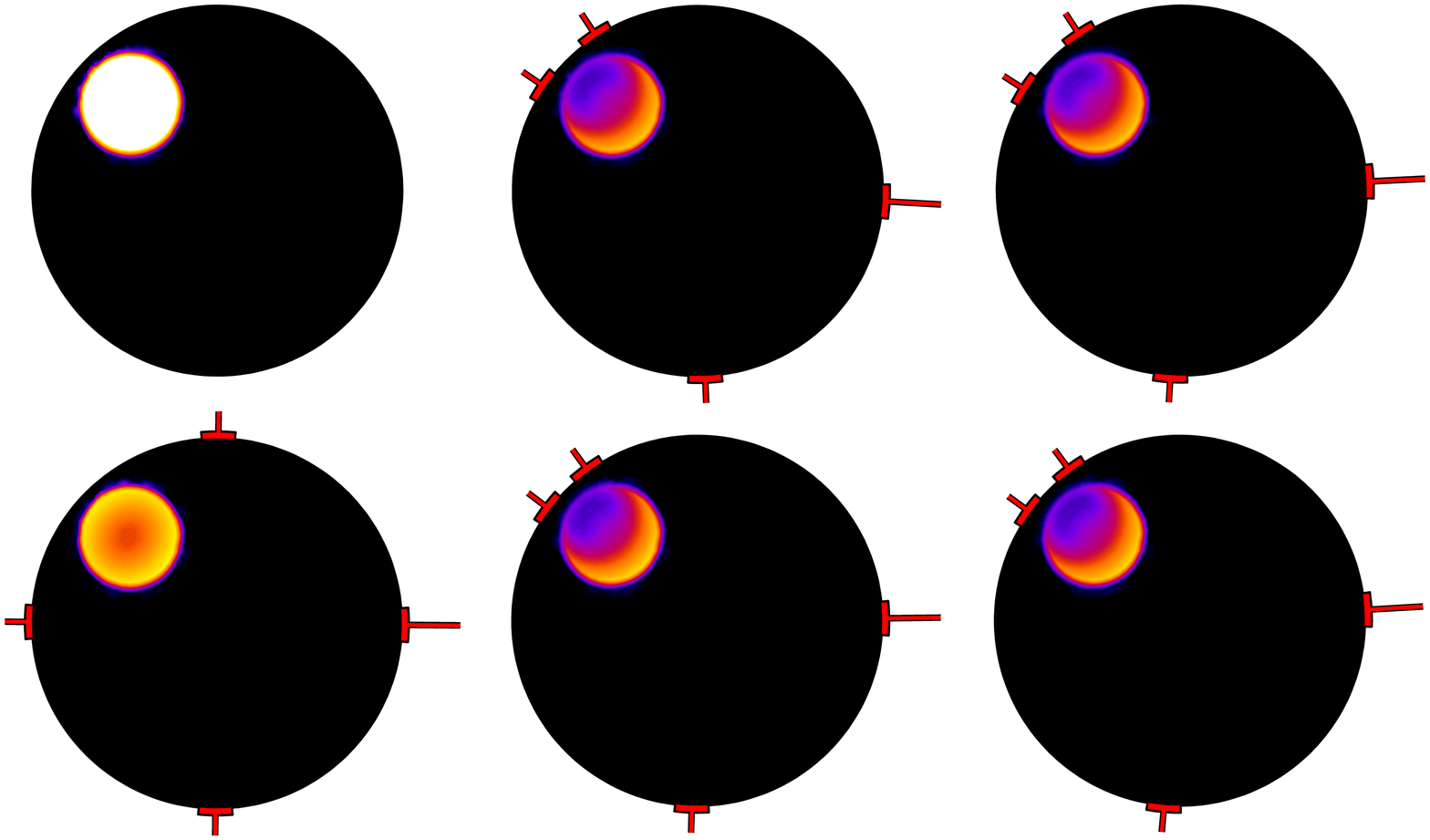}\\[8pt]
\includegraphics[width=0.5\textwidth]{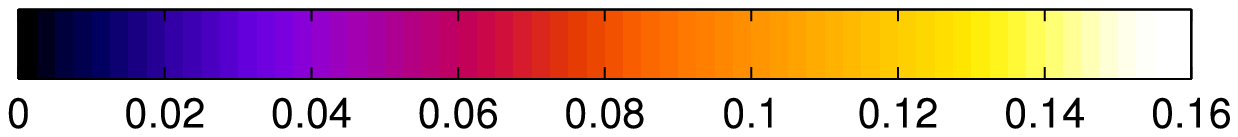} 
\caption{\label{fig:brutevsalg} Comparison of the output of Algorithm \ref{alg:steep} with the `brute force' minimizer in Case 1. The colormap corresponds to the 
point variances (diagonal of the covariance matrix) of the distribution in question. Top left: the prior variance~\eqref{eq:inhomog}. Bottom left: the initial guess $\theta^-_{\rm init}$ with the corresponding posterior variance. Center column: the global minimizer for ${\rm tr}(\Gamma_\ast)$  (top) and $\log \det (\Gamma_\ast)$ (bottom) in \eqref{eq:cost} with the corresponding posterior variances.  Right column:  the output of Algorithm~\ref{alg:steep} for ${\rm tr}(\Gamma_\ast)$  (top) and $\log \det (\Gamma_\ast)$ (bottom) in \eqref{eq:cost} with the corresponding  posterior variances.}
\end{center}
\vspace{-2mm}
\end{figure}

In the first test, the object of interest $\Omega$ is a unit disk and the prior covariance is inhomogeneous: We pick a smaller disk $\Omega' \subset \Omega$ (cf.~Fig.~\ref{fig:brutevsalg}) and choose $\Gamma_{\rm pr} = \Gamma(0.5,\kappa)$ with
\begin{equation}\label{eq:inhomog}
\kappa_{ij} = \begin{cases} 0.4, & x^{(i)},x^{(j)} \in \overbar{\Omega}', \\ 0.03, & x^{(i)},x^{(j)} \notin \overbar{\Omega}', \\ 0, & x^{(i)}\in \overbar{\Omega}', \ x^{(j)}\notin \overbar{\Omega}', \ {\rm or} \ {\rm vice \ versa}, \end{cases} \qquad x^{(i)},x^{(j)} \in D.
\end{equation}
In other words, $\Omega$ consists of two uncorrelated parts, one of which is a relatively small inclusion with a large variance, whereas the background conductivity values are `almost known'. The four employed  electrodes are depicted in Fig.~\ref{fig:brutevsalg}, with the special, current-feeding electrode indicated by a longer `cord'.


As we are primarily interested in whether Algorithm \ref{alg:steep} finds the global minimum of the functional $\psi$ from \eqref{eq:cost}, we evaluate $\psi$ on a reasonably large set of angular parameters $\theta^-\in [0,2\pi)^4$ (so that the ordering of the electrodes does not change) and eventually minimize over the evaluations. In Fig.~\ref{fig:brutevsalg}, the brute force based global minimizers are compared to the outputs of Algorithm~\ref{alg:steep}, with uniformly spaced electrodes serving as the initial guess. At least with such a low number of electrodes, the numerical results support the functionality of Algorithm \ref{alg:steep}, that is, for both alternative functionals in \eqref{eq:cost}, the optimal configurations provided by Algorithm~\ref{alg:steep} and the brute force computations approximately coincide. Moreover, the optimal positions seem rather logical: Two of the four electrodes move close to the smaller disk that carries most of the uncertainty, which clearly reduces the variance of the posterior, as illustrated in Fig.~\ref{fig:brutevsalg}.

To demonstrate the nontrivial dependence of the optimal measurement configuration on the prior knowledge about the conductivity, Fig.~\ref{fig:test4el} illustrates how the optimal electrode positions change when the covariance factor for $x^{(i)}, x^{(j)} \notin \overbar{\Omega}'$ in~\eqref{eq:inhomog} varies in the interval $\kappa_{ij} \in [0.05, 0.2]$. In other words, the images in Fig.~\ref{fig:test4el} correspond to different ratios between the prior uncertainties in $\overline{\Omega}'$ and $\Omega \setminus \Omega'$. Here, we employ the A-optimality criterion, i.e., ${\rm tr}(\Gamma_\ast)$ in \eqref{eq:cost}. Notice that with the two highest values $\kappa_{ij} = 0.15, 0.2$ for $x^{(i)}, x^{(j)} \notin \overbar{\Omega}'$, the electrode positions presented in Fig.~\ref{fig:test4el} are not symmetric with respect to $\Omega'$, meaning that there exist two optimal configurations that are certain mirror images of one another. 

\begin{figure}[t!]
\begin{center}
\includegraphics[width=1\textwidth]{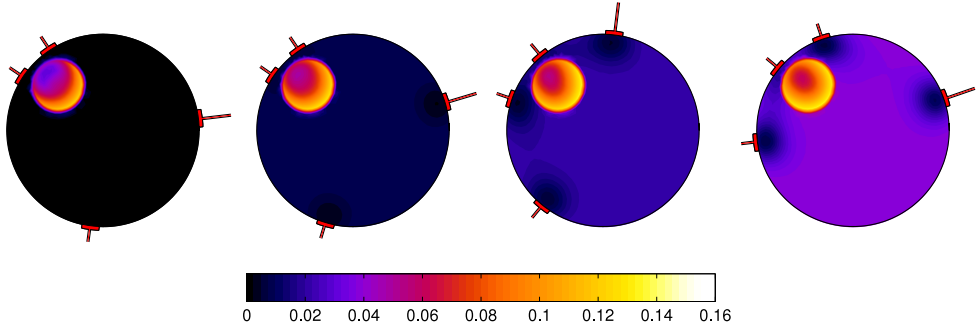}\\[10pt]
\includegraphics[width=0.5\textwidth]{096617Rf1.eps} 
\caption{\label{fig:test4el} Illustration of the effect of $\kappa_{ij}$ for $x^{(i)}, x^{(j)} \notin \overbar{\Omega}'$ in \eqref{eq:inhomog} on the optimized electrode positions in Case 1 with the A-optimality criterion,~i.e.,~with ${\rm tr}(\Gamma_\ast)$ in \eqref{eq:cost}. From left to right: the optimized electrode positions corresponding to the values $\kappa_{ij} = 0.05,0.1,0.15$ and $0.2$ for the `prior background standard deviation' in \eqref{eq:inhomog}. The colormap is as in Fig~\ref{fig:brutevsalg}.}
\end{center}
\vspace{-2mm}
\end{figure}

\subsection*{Case 2: Semidisks with different prior variances}

In this example, there are twelve electrodes attached to the unit disk $\Omega = B(0,1)$. The prior covariance is of the form $\Gamma_{\rm pr} = \Gamma(0.5,\kappa)$ where 
\begin{equation}\label{eq:split}
\kappa_{ij} = \begin{cases} 0.03, & x_2^{(i)},x_2^{(j)} \geq 0, \\[2pt] 0.4, & x_2^{(i)},x_2^{(j)} < 0, \\[1pt] 0, & x_2^{(i)} \geq 0, \ x_2^{(j)} < 0, \ {\rm or} \ {\rm vice \ versa},\end{cases}
\qquad x^{(i)},x^{(j)} \in D,
\end{equation}
that is, the upper and lower halves of $\Omega$ are uncorrelated, with the point variances being considerably higher in the lower than in the upper half. 
The outputs of Algorithm~\ref{alg:steep} for the two objective functionals in \eqref{eq:cost} are illustrated in Fig.~\ref{fig:12el_half}. We observe that both optimization criteria yield qualitatively very similar results, which are also intuitively acceptable: Almost all electrodes are moved around the lover half of $\Omega$ where the uncertainty about the conductivity is the highest {\em a priori}. Take note that the slight asymmetry in the optimal electrode configurations of Fig.~\ref{fig:12el_half} is probably due to the original/final position of the current-feeding electrode (with a longer cord); the reached optimum is not necessarily the global one, albeit the corresponding value of the objective functional is probably very close to optimal.

\begin{figure}[b!]
\begin{center}
\includegraphics[width=1\textwidth]{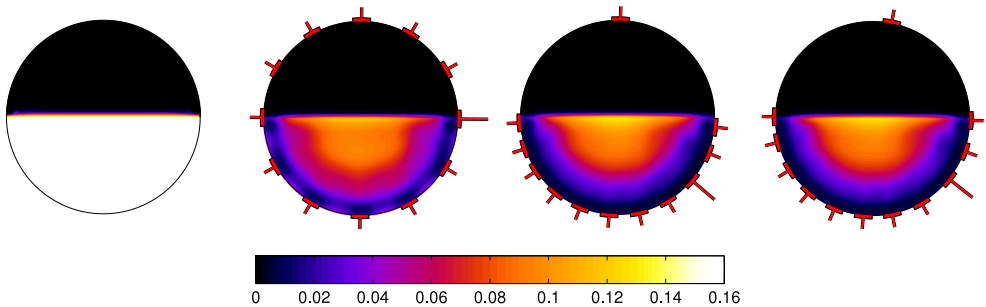}\\[8pt]
\includegraphics[width=0.5\textwidth]{096617Rf1.eps}
\caption{\label{fig:12el_half} Output of Algorithm \ref{alg:steep} with twelve electrodes 
in Case 2. 
The colormap corresponds to the point variances of the distribution in question. From left to right: (i) The prior variance \eqref{eq:split}. (ii) The initial guess $\theta^-_{\rm init}$ with the corresponding posterior variance. (iii)~\&~(iv) The outputs of Algorithm~\ref{alg:steep} with the corresponding posterior variances for $\log \det (\Gamma_\ast)$ and ${\rm tr}(\Gamma_\ast)$ in \eqref{eq:cost}, respectively. 
}
\end{center}
\vspace{-2mm}
\end{figure}

We proceed with an evaluation of the output of Algorithm \ref{alg:steep}. More precisely, we compare the mean square error in the MAP reconstructions corresponding to the optimized electrode positions with that for the initial equidistant electrode configuration. The outline of the procedure is as follows: A number of $ N_{\rm draw} $ conductivities $ \sigma_{\rm draw} $ are drawn from the prior $ \mathcal{N}(\sigma_\ast,\Gamma_{\rm pr}) $. For each $ \sigma_{\rm draw} $, a datum $ \Vcal_{\rm sim} $ is simulated by adding artificial noise drawn from $\mathcal{N}(0,\Gamma_{\rm noise})$ to the forward solution (computed on a denser FE mesh than the one used for the reconstructions to avoid an inverse crime). Take note that the chosen noise covariance is such that the expected $|\Vcal_{\rm sim}-\Rcal(\sigma_{\rm draw})\Ical|^2$ is about $ 0.005^2 |\Rcal(\sigma_\ast)\Ical|^2$, i.e., it roughly corresponds to $0.5 \%$ of relative error. Subsequently, for each $ \Vcal_{\rm sim} $, an (approximate) MAP estimate, determined by \eqref{map} where the posterior density  is the unlinearized one from \eqref{eq:gausspost}, is computed by a variant of the Gauss--Newton algorithm (see, e.g.,~\cite{Darde13a}). 

\begin{figure}[t]
\begin{center}
\includegraphics[width=1\textwidth]{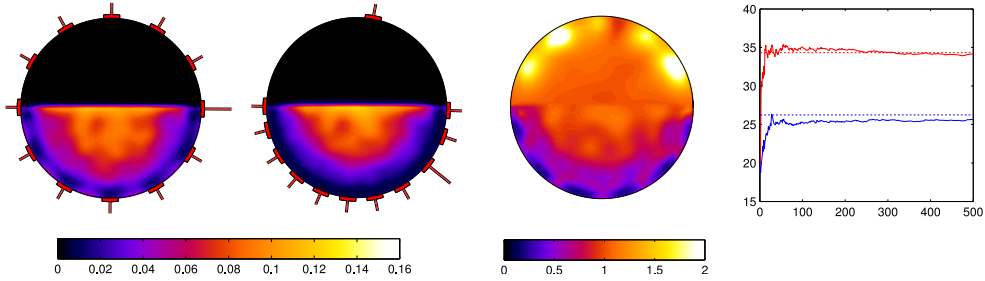}\\[16pt]
\caption{\label{fig:12el_half_MAP}
An evaluation of the output of Algorithm \ref{alg:steep} in Case 2. 
From left to right: (i) The average of the pointwise squared errors $(\sigma_{\rm draw}-\hat{\sigma}_{\rm MAP}(\Vcal_{\rm sim};\theta^ -))^2$ for an equidistant $\theta^- = \theta_{\rm init}^-$. (ii) The average of the pointwise squared errors $(\sigma_{\rm draw}-\hat{\sigma}_{\rm MAP}(\Vcal_{\rm sim};\theta^ -))^2$ for an optimized $\theta^- = \theta_{\rm opt}^-$ with ${\rm tr}(\Gamma_\ast(\theta))$ in \eqref{eq:cost}.
(iii) The ratio of (ii) and (i). (iv) The value of \eqref{eq:mse} approximated by ${\rm tr}(\Gamma_\ast(\theta))$ (dashed horizontal line) and by the random draws (solid line as a function of the number of random draws), respectively; the upper (red) and lower (blue) pairs of graphs correspond to $\theta_{\rm init}^-$ and $\theta_{\rm opt}^-$, respectively.
} 
\end{center}
\vspace{-2mm}
\end{figure}

Assuming that the (pseudo)random draws are distributed as intended, it is easy to 
deduce that each pair $ (\sigma_{\rm draw},\Vcal_{\rm sim}) $ is a realization of a random variable with the joint density $ p(\sigma,\Vcal) = p(\Vcal | \sigma) p_{\rm pr}(\sigma) $ (cf.~\eqref{eq:gausspost}). By the {\em strong law of large numbers}, the average of the square errors $ |\sigma_{\rm draw}-\hat{\sigma}_{\rm MAP}(\Vcal_{\rm sim};\theta^ -)|^2$ thus tends almost surely to \eqref{eq:mse} as $ N_{\rm draw}$ goes to infinity, with $ A $ being the identity matrix in \eqref{eq:mse}. The numerical results with $ N_{\rm draw} = 500 $ are visualized in Fig.~\ref{fig:12el_half_MAP}. In particular, we get approximately $0.75 $ as the ratio between the averages of $ |\sigma_{\rm draw}-\hat{\sigma}_{\rm MAP}(\Vcal_{\rm sim};\theta_{\rm opt}^ -)| ^2 $ and $ |\sigma_{\rm draw}-\hat{\sigma}_{\rm MAP}(\Vcal_{\rm sim};\theta_{\rm init}^ -)|^2 $, which demonstrates the superiority of the optimized configuration. 
 


\subsection*{Case 3: Objects with different boundary shapes}
The final numerical test considers finding the optimal locations of twelve electrodes around three different objects: the ellipse-like shape, the peanut, and the complicated domain shown in Fig.~\ref{fig:12el_eccentric}. For all three objects, the conductivity prior is $\Gamma(0.5,0.4)$ and the initial electrode positions for Algorithm~\ref{alg:steep} are presented in the top row of Fig.~\ref{fig:12el_eccentric}. The optimal electrode locations produced by Algorithm~\ref{alg:steep} (with the choice $\log \det (\Gamma_\ast)$ in the functional \eqref{eq:cost}) are illustrated on the bottom row of Fig.~\ref{fig:12el_eccentric}. In the optimized configurations, the widest gaps end up over boundary segments with negative curvature, but other general conclusions are difficult to draw based on Fig.~\ref{fig:12el_eccentric} only. As an example, the optimal configuration for the ellipse seems to be almost equidistant. Be that as it may, the values of the objective functional \eqref{eq:cost} are considerably lower for the optimized electrode configurations compared to their unoptimized counterparts. 


\begin{figure}[h!]
\begin{center}
\includegraphics[width=1\textwidth]{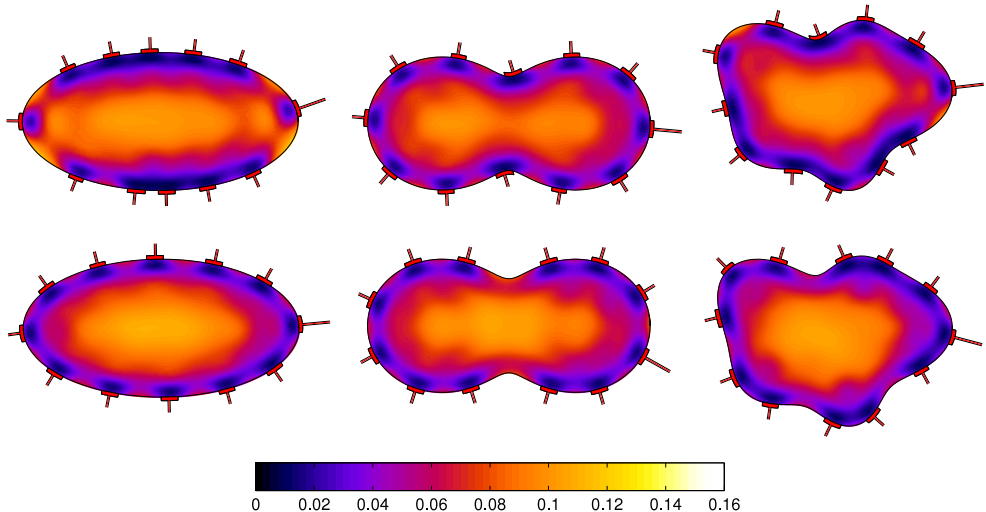}\\[8pt]
\includegraphics[width = 0.5\textwidth]{096617Rf1.eps}
\caption{\label{fig:12el_eccentric} Output of Algorithm \ref{alg:steep} with different object shapes and a homogeneous Gaussian smoothness prior in Case 3. 
The colormap corresponds to the point variances of the distribution in question.
Top row: the initial guesses $\theta_{\rm init}^-$ with the corresponding posterior variances. 
Bottom row: outputs of Algorithm~\ref{alg:steep} for $\log\det(\Gamma_\ast)$ in \eqref{eq:cost} with the corresponding posterior variances.}
\end{center}
\vspace{-2mm}
\end{figure}

\section{Conclusions}
\label{sec:conc}
We have proposed a methodology for the optimization of the electrode positions in EIT. The employed optimality criteria were derived from the Bayesian approach to inverse problems, with the aim being to make the posterior density of the conductivity as localized as possible.
In order to lighten the computational load, we approximated the posterior by linearizing the current-to-voltage map of the CEM with respect to the conductivity around the prior mean, which allowed closed form expressions for the to-be-minimized quantities: the trace and the determinant of the posterior covariance matrix, which correspond to the so-called A- and D-optimality criteria of experiment design, respectively. 

The introduced optimization algorithm is of the steepest descent type; the needed derivatives with respect to the electrode locations were computed based on appropriate Fr\'echet derivatives of the CEM. Our numerical experiments demonstrate (i) the functionality of the algorithm, that (ii) the optimal electrode configurations are nontrivial even in relatively simple settings, and that (iii) the employment of the optimal electrode locations improves the quality of EIT reconstructions.


\bibliographystyle{acm}
\bibliography{opel-refs}{}

\end{document}